\documentclass[12pt, reqno]{amsart}
\usepackage{mathrsfs}
\usepackage{amsfonts,amsmath,dsfont,amssymb,amsthm,stmaryrd,bbm,color,comment}
\usepackage[hidelinks]{hyperref}
\usepackage{appendix}
\usepackage[pdftex]{graphicx}
\usepackage{bbm}
\usepackage{array}
\newcolumntype{L}{>{\centering\arraybackslash}m{2.5cm}}
\usepackage{multirow}
\usepackage{xcolor}
\usepackage[percent]{overpic}
\usepackage{algorithm}
\usepackage{algpseudocode}
\usepackage{mathtools}

\newtheorem{thm}{Theorem}[section]
\newtheorem{lem}[thm]{Lemma}
\newtheorem{prop}[thm]{Proposition}

\theoremstyle{definition}

\newtheorem{conj}[thm]{Conjecture}

\theoremstyle{remark}
\newtheorem{rmk}[thm]{Remark}
\numberwithin{equation}{section}

\newcommand{\supp}{\mathrm{supp}}

\newcommand{\R}{\mathbb{R}}
\newcommand{\N}{\mathbb{N}}
\newcommand{\Z}{\mathbb{Z}}

\def\P{{\mathbb P}}

\begin{document}
\setcounter{page}{1}

\color{black}{

\centerline{}
\centerline{}
\title[A reverse EPI for i.i.d. log-concave random variables]{A reverse entropy power inequality for i.i.d. log-concave random variables}
\author[Zhen Fu, and Jiange Li]{Zhen Fu, and Jiange Li} 
\address{(Z.F, J.L) Institute for Advanced Study in Mathematics, Harbin Institute of Technology, China}
\email{zhenfu@stu.hit.edu.cn, jiange.li@hit.edu.cn}


\begin{abstract}
We show that $h_\infty(X+Y)\leq h_\infty(Z+W)$, where $X, Y$ are independent log-concave random variables, and $Z, W$ are exponential random variables having the same respective $\infty$-R\'enyi entropies. Analogs for integer-valued monotone log-concave random variables are also obtained. Our main tools are decreasing rearrangement, majorization, and the change of measure.
\end{abstract} 

\maketitle


\section{Introduction}
Let $X$ be a random vector taking values in $\R^d$. Suppose that it has density $f$ with respect to the Lebesgue measure. The classical \textit{differential entropy} (also called \textit{Boltzmann-Shannon entropy}) of $X$ is defined as
$$
h(X) =-\int_{\R^d} f(x)\log f(x)dx.
$$
The celebrated \textit{entropy power inequality} (or EPI for short) of Shannon \cite{Sha48} (and Stam \cite{Sta59}) states that for independent random vectors $X$ and $Y$ in $\R^d$ such that the entropies of $X, Y$, and $X+Y$ exist, it holds that
\begin{equation}\label{eq:epi}
e^{\frac{2}{d}h(X+Y)}\geq e^{\frac{2}{d}h(X)}+e^{\frac{2}{d}h(Y)},
\end{equation}
where equality holds if and only if $X$ and $Y$ are Gaussian random vectors with proportional covariance matrices. The EPI found its original applications in the study of channel capacity in information theory. It has been known very well that EPI is closely related to geometric and functional inequalities in a variety of mathematical fields, maybe most notably the Brunn-Minkowski inequality (or BMI for short) in convex geometry. We refer the interested reader to \cite{Gar02, MMX17} for profound connections among EPI, BMI, and many other geometric and functional inequalities.

Convexity (or concavity) underpins the validity of the reversal of many geometric and functional inequalities. A function $f:\R^d\to \R_+$ is called \textit{log-concave} if it can be written as $f(x)=e^{-V(x)}$, where $V: \R^d\to \R\cup\{\infty\}$ is a convex function. We say that a random vector $X$ taking values in $\R^d$ is log-concave provided that its density is a log-concave function. (We also say that $X$ has log-concave distribution). The class of log-concave functions can be thought of as functional analogs of convex bodies. They are prominent objects studied in convex geometry, high dimensional probability and statistics, as well as theoretical computer science. The class of log-concave distributions include a wide range of important probability distributions, such as the Gaussian distributions, uniform distributions on convex bodies, and exponential distributions. 

The reversal of EPI was first addressed by Bobkov and Madiman \cite{BM12}. They showed that for independent log-concave random vectors, under linear volume preserving maps, inequality \eqref{eq:epi} can be reversed at the cost of a constant factor on the r.h.s of the inequality. This can be thought of as a functional analog of Milman's reverse BMI for convex bodies \cite{Mil86}. It is necessary to place log-concave random vectors ``at right positions" via linear volume preserving maps for the reversed EPI to hold. We are interested in the sharp reverse EPI for \textit{independent and identically distributed} (or i.i.d. for short) log-concave random vectors. In this case, linear volume preserving maps are not needed and the question can be phrased as follows.

\begin{conj} [Folklore] \label{conj:folklore}
Let $X$ and $Y$ be i.i.d. log-concave random vectors taking values in $\R^d$. The entropy increment $h(X+Y)-h(X)$ is maximized when $X$ and $Y$ have exponential distributions. 
\end{conj}

Clearly, the entropy increment is affine invariant. An earlier result of Cover and Zhang \cite{CZ94} shows that log-concavity and identical distribution (without independence between $X$ and $Y$) yield the upper bound $d\log 2$ of the entropy increment. Huang, Slomka, Tkocz, and Vritsiou (\cite{HSTV22}, Section 6) observed that the improvement of Cover and Zhang's bound for Conjecture \ref{conj:folklore} will lead to important progress on Hadwiger's covering problem. Ball, Nayar, and Tkocz \cite{BNT16} studied the entropy increment for a two dimensional log-concave random vector $(X, Y)$, whose marginals have identical differential entropy (but not necessarily independent). This variant of reverse EPI and its generalizations are entropic analogs of Busemann's convexity theorem of intersection bodies \cite{Bus49} and the convexity of $p$-cross-section bodies conjectured by Gardner and Giannopoulos \cite{GG99} (the end of Section 5, and the discussion in \cite{Li18}).

Aforementioned forward and reverse EPIs have been studied for general R\'enyi entropies. For $p\in (0, 1)\cup (1, \infty)$, the order $p$ R\'enyi entropy of $X$ is defined as 
\begin{equation}\label{eq:renyi}
h_p(X)=\frac{1}{1-p}\log \int_{\R^d}f(x)^pdx.
\end{equation}
For $p\in\{0, 1, \infty\}$, by taking limits, we have 
$h_0(X)=\log |\supp\{f\}|$, $h_\infty(X) =-\log \|f\|_\infty$, and $h_1(X) =h(X)=-\int_{\R^d} f(x)\log f(x)dx$. Here, we denote by $|\supp\{f\}|$ the Lebesgue measure of the support of $f$, and $\|f\|_\infty$ is the essential supremum of $f$, and $h_1(X)$ is the classical differential entropy $h(X)$ of $X$. 

It is suspected that a transition will occur in the increment of R\'enyi entropy. More precisely, \textit{there exists an absolute constant $p_0\in (0, 1)$ such that $h_p(X+Y)-h_p(X)$ is maximized by exponential random vectors for $p\geq p_0$, and is maximized by uniform random vectors taking values in the cube $[0, 1]^d$ for $0<p\leq p_0$}. It is perhaps worth to point out that the increment of R\'enyi entropy of order $p>1$ (respectively, $0<p<1$) boils down to the minimization (respectively, maximization) of $\|f\ast f\|_p/\|f\|_p$ over the class of integrable log-concave functions. This can be thought of as a special type of reverse Young's convolution inequality.  Our main result reads as follows.

\begin{thm}\label{thm:reverse-rogozin}
Let $X$ and $Y$ be independent real-valued log-concave random variables. Let $Z$ and $W$ be independent exponential random variables such that $h_\infty(Z)=h_\infty(X)$ and $h_\infty(W)=h_\infty(Y)$. We have
\begin{align*}
h_\infty(X+Y)\leq h_\infty(Z+W).
\end{align*}
A discrete analog holds for integer-valued monotone log-concave random variables.
\end{thm} 

\begin{rmk}
If $X_1, \cdots, X_n$ are independent real-valued random variables, and $U_1, \cdots, U_n$ are independent uniform random variables such that $h_\infty(U_i)=h_\infty(X_i)$ for $i=1,\cdots, n$, Rogozin \cite{Rog87} proved that
$$
h_\infty(X_1+\cdots+X_n)\geq h_\infty(U_1+\cdots+U_n).
$$
Hence, our result can also be seen as a reverse Rogozin type result for two log-concave random variables. However, we are not able to establish analogs of Theorem \ref{thm:reverse-rogozin} for three or more log-concave random variables or in multi-dimensions.
\end{rmk}

Another closely related problem raised by Madiman and Kontoyiannis \cite {MK18} is the characterization of log-concave distributions that maximize the R\'enyi entropic increment $h_p(X-Y)-h_p(X)$. This resembles the Rogers-Shephard inequality in convex geometry \cite{RS57}. In a certain sense, the symmetry of $X-Y$ (implicitly) makes this problem more tractable than its counterpart of the $X+Y$ case. Melbourne and Tkocz (\cite{MT21}, Theorem VI.1) showed that the increment of R\'enyi entropy of order $p\geq 2$ is indeed maximized by exponential random vectors. A discrete analog was obtained by Melbourne and Palafox-Castillo  (\cite{MPC23}, Theorem 2.11) for integer-valued monotone log-concave random variables. Very recently, Han, Zhan, Zhou, and Zhu \cite{HZZZ26} obtained sharp estimate for all orders $p\ge 1$, as well as a complete characterization of the maximizers.


\section{Rearrangement and majorization}

Rearrangement and majorization are classical tools for establishing sharp geometric and functional inequalities. This section provides the necessary preliminaries on decreasing rearrangement and majorization theory in one dimension. Readers already familiar with rearrangement theory are encouraged to skip this section.

\subsection{Decreasing rearrangement}
Let $A\subseteq \R$ be a Borel measurable set with Lebesgue measure $|A|$. The \textit{decreasing rearrangement} of $A$ is defined as $A^\downarrow =[0, |A|)$. The following properties of decreasing rearrangement may be scattered in the literature. We provide proof of them for the convenience of the reader.

\begin{prop}\label{prop:rearrange propt-set}
Let $A, B\subseteq\R$ be Borel measurable sets.
\begin{enumerate}
\item For all $\lambda\geq 0$, it holds that $(\lambda A)^\downarrow=\lambda A^\downarrow$.
\item  It holds that $A^\downarrow+B^\downarrow \subseteq (A+B)^\downarrow$.
\item  Let $\{A_n\}_{n=1}^\infty$ be a sequence of measurable sets such that $A_n\subseteq A_{n+1}$ for all $n=1, 2\cdots$. It holds that
$(\cup_{n=1}^\infty A_n)^\downarrow=\cup_{n=1}^\infty A_n^\downarrow$.
\end{enumerate}
\end{prop}

\begin{proof} 
Statement (1) readily follows from the definition. By definition, $A^\downarrow=[0, |A|)$, $B^\downarrow=[0, |B|)$, and $(A+B)^\downarrow =[0, |A+B|)$. The Brunn-Minkowski inequality in one dimension yields $|A+B|\geq |A|+|B|$. Hence, $A^\downarrow+B^\downarrow=[0,|A|+|B|) \subseteq (A+B)^\downarrow$. This proves statement (2).
 
By definition, $A_n^\downarrow=[0, |A_n|)$. Since $A_n\subseteq A_{n+1}$, we have
$$
\cup_{n=1}^\infty A_n^\downarrow=\big[0, \lim_{n\rightarrow\infty} |A_n|\big)=\left[0, |\cup_{n=1}^\infty A_n|\right)=(\cup_{n=1}^\infty A_n)^\downarrow.
$$
This proves the last statement.
\end{proof}

Let $f: \R\rightarrow \R\cup\{\pm \infty\}$ be a Borel measurable function such that all super level sets $\{f>\lambda\}$ have finite measures (which is called \textit{vanishing at infinity}). We define the \textit{decreasing rearrangement} $f^\downarrow: \R_+\rightarrow \R\cup\{\pm \infty\}$ as 
\begin{equation}\label{eq:mr-f}
f^\downarrow(x)=\sup\left\{\lambda\in\R: x\in \{f>\lambda\}^\downarrow\right\}.
\end{equation}

\begin{prop}\label{prop:rearrange propt-function}
Let $f: \R\rightarrow\R\cup\{\pm \infty\}$ be a Borel measurable function.
\begin{enumerate}
\item For all $\lambda\in\R$, it holds that $\{f^\downarrow>\lambda\}=\{f>\lambda\}^\downarrow$. (Hence, this gives an equivalent definition of the decreasing rearrangement $f^\downarrow$).
\item If $f$ is concave, then $f^\downarrow$ is also concave. Consequently,  if $f$ is log-concave, then $f^\downarrow$ is also log-concave.
\end{enumerate}
\end{prop}

\begin{proof}
For each $x\in\{f^\downarrow>\lambda\}$, by the definition of $f^\downarrow$ in \eqref{eq:mr-f}, there exists $\lambda_x>\lambda$ such that $x\in \{f>\lambda_x\}^\downarrow\subseteq \{f>\lambda\}^\downarrow$. This shows that $\{f^\downarrow>\lambda\}\subseteq \{f>\lambda\}^\downarrow$. To see the other direction, we apply Proposition \ref{prop:rearrange propt-set} (statement (3)) to write
$$
 \{f>\lambda\}^\downarrow=\left(\cup_{n=1}^\infty \{f>\lambda+1/n\}\right)^\downarrow=\cup_{n=1}^\infty \{f>\lambda+1/n\}^\downarrow.
$$
Hence, for each $x\in \{f>\lambda\}^\downarrow$, there exists $n\in\N$ such that $x\in \{f>\lambda+1/n\}^\downarrow$. By the definition of $f^\downarrow$ in \eqref{eq:mr-f}, this yields $f^\downarrow(x)\geq \lambda+1/n>\lambda$, and hence, $x\in \{f^\downarrow>\lambda\}$. This completes the proof of the first statement.

Suppose $f$ is concave. This is equivalent to the following set-theoretic inequality.  For any $\lambda_0, \lambda_1\in \R$ and $t\in [0, 1]$, it holds that
\begin{equation}\label{eq:concave}
(1-t)\{f>\lambda_0\}+t\{f>\lambda_1\}\subseteq \{f >(1-t)\lambda_0+t\lambda_1\}.
\end{equation}
Next we prove the inclusion inequality \eqref{eq:concave} for $f^\downarrow$. This will yield the concavity of $f^\downarrow$. We apply statement (1) of the current proposition and Proposition \ref{prop:rearrange propt-set} (statements (1) and (2))  to obtain
\begin{align*}
(1-t)\{f^\downarrow>\lambda_0\}+t\{f^\downarrow>\lambda_1\} &= (1-t)\{f>\lambda_0\}^\downarrow+t\{f>\lambda_1\}^\downarrow\\
&= ((1-t)\{f>\lambda_0\})^\downarrow+(t\{f>\lambda_1\})^\downarrow\\
&\subseteq ((1-t)\{f>\lambda_0\}+t\{f>\lambda_1\})^\downarrow\\
& \subseteq \{f>(1-t)\lambda_0+t\lambda_1\}^\downarrow\\
&=\{f^\downarrow >(1-t)\lambda_0+t\lambda_1\}.
\end{align*}

Suppose $f(x)=e^{V(x)}$ is log-concave, i.e., $V(x)$ is concave. Write $f^\downarrow(x)=e^{\tilde{V}(x)}$. For all $\lambda\in\R$, we have
$$
\{\tilde{V}>\lambda\}=\{f^\downarrow>e^{\lambda}\}=\{f>e^{\lambda}\}^\downarrow=\{V>\lambda\}^\downarrow.
$$
Using the first statement, we conclude that $\tilde{V}=V^\downarrow$. Since $V$ is concave, $V^\downarrow$, and hence, $\tilde{V}$ are concave. (One can show the concavity of $\tilde{V}$ by establishing the corresponding set-theoretic inequality). This yields the log-concavity of $f^\downarrow$.
\end{proof}

\subsection{Majorization}

Let $f, g: \R\to \R_+$ be two integrable functions such that $$\int_{\R}f(x)dx=\int_{\R}g(x)dx\textcolor{red}{.}$$ We say that $f$ is \textit{majorized} by $g$ if it holds for all $t\geq 0$ that
\begin{equation}\label{eq:maj-def}
\int_{\R}(f(x)-t)_+dx\leq \int_{\R}(g(x)-t)_+dx.
\end{equation}
For $t\geq 0$, we write $m_f(t)=|\{x\in\R: f(x)> t\}|$ and define $m_g(t)$ in a similar manner. Then we can rewrite \eqref{eq:maj-def} as
$$
\int_t^\infty m_f(s)ds\leq \int_t^\infty m_g(s)ds.
$$ 

A real-valued function $f$ on $\R$ is called \textit{unimodal} if there exists a point $x_0$ such that $f$ is non-decreasing for $x<x_0$ and non-increasing for $x>x_0$. Clearly, log-concave functions are unimodal. The following result provides a sufficient condition for the majorization of $g$ over $f$. 

\begin{prop}\label{prop:maj-suf}
Suppose that for any Borel measurable set $A$ of finite measure, there exists a Borel measurable set $B$ with equal measure as $A$ such that
\begin{equation}\label{eq:maj-suf}
\int_Af(x)dx\leq \int_Bg(x)dx.
\end{equation}
Then $f$ is majorized by $g$. If $f$ and $g$ are unimodal, it suffices to verify the above inequality for intervals.
\end{prop}

\begin{proof}
For any given $t\in\R_+$, we write $A=\{f>t\}$. Then there exists $B\subseteq\R$ of the same measure as $A$ such that
\begin{align*}
\int_{\R} (f(x)-t)_+dx &=\int_A (f(x)-t)dx\\
&\leq  \int_B (g(x)-t)dx\\
& \leq\int_B (g(x)-t)_+dx\\
&\leq \int_{\R} (g(x)-t)_+dx.
\end{align*}
This shows the majorization of $g$ over $f$. 

Note that the integral of a unimodal function over all measurable sets of fixed measure attains its maximum for an interval. Therefore, if $f$ is unimodal, it is sufficient to evaluate the integral on the left-hand side of inequality \eqref{eq:maj-suf} over intervals. Moreover, when $g$ is also unimodal, one often considers the integral on the right-hand side of inequality \eqref{eq:maj-suf} over intervals, even though considering general Borel measurable sets with the same measure as  $A$ would suffice.
\end{proof}

We also need the following result in majorization theory. A preliminary version was first proved by Hardy, Littlewood, and P\'olya \cite{HLP29} and various extensions were discussed in Chong \cite{Cho74}. 

\begin{prop} [Proposition 7.3, \cite{WM14}]\label{proposition:major}
Let $\phi: \mathbb{R}_+\to \mathbb{R}$ be a convex function such that $\phi(0)=0$ and it is continuous at 0. If $f$ and $g$ are probability densities such that $f$ is majorized by $g$, then 
\begin{equation}\label{eq:maj}
\int_{\mathbb{R}}\phi(f(x))dx\leq \int_{\mathbb{R}}\phi(g(x))dx,
\end{equation}
provided that both sides are well defined.
\end{prop}

\section{Main Results}

\begin{prop}\label{prop:2-ent-jump}
Let $X$ and $Y$ be i.i.d. log-concave random vectors in $\R^d$. We have
\begin{equation}\label{eq:2-ent-jump}
h_2(X+Y)\leq h_2(X)+d\log 2.
\end{equation}
Moreover, the inequality can not be improved.
\end{prop}

\begin{rmk}
One can check that equality \eqref{eq:2-ent-jump} can be attained by standard exponential random vectors $X$ with density $e^{-(x_1+\cdots+x_d)}{1}_{\mathbb{R}^d_+}(x)$ and its invertible images $TX$, where $T$ is an invertible real $d\times d$ matrix.
\end{rmk}

\begin{proof}
Let $\tilde{X}$ and $\tilde{Y}$ be independent copies of $X$ and $Y$, respectively. Let $f$ be the density of $X-\tilde{X}$. Since $f$ is symmetric and log-concave, we have
$$
\|f\ast f\|_\infty=(f\ast f)(0)=\int_{\R^d}f(x)^2dx\geq f(0)\int_{\R^d}f(2x)dx=\frac{f(0)}{2^d}.
$$ 
This can be rewritten as
$$
h_\infty((X-\tilde{X})+(Y-\tilde{Y}))\leq  h_\infty(X-\tilde{X})+d\log 2.
$$
For any i.i.d. log-concave random vectors $U$ and $V$, it holds that $h_\infty(U-V)=h_2(U)$. This observation, together with the inequality above, implies that 
\begin{align}\label{eq:symmetrization}
h_2(X+Y) &=h_\infty((X+Y)-(\tilde{X}+\tilde{Y})) \notag\\
&=h_\infty((X-\tilde{X})+(Y-\tilde{Y})) \\
&\leq h_\infty(X-\tilde{X})+d\log 2 \notag\\
&= h_2(X)+d \log 2 \notag.
\end{align}
This completes the proof.
\end{proof}

\begin{rmk}\label{rmk:sigma2}
Note that $X-\tilde{X}$ and $Y-\tilde{Y}$ are symmetric and independent. Then, we can proceed with \eqref{eq:symmetrization} to obtain
\begin{align*}
h_2(X+Y) &=h_\infty((X-\tilde{X})+(Y-\tilde{Y}))\\
&=h_\infty((X-\tilde{X})-(Y-\tilde{Y}))\\
&=h_\infty((X-Y)-(\tilde{X}-\tilde{Y}))\\
&=h_2(X-Y).
\end{align*}
Hence, inequality \eqref{eq:2-ent-jump} is the same as the $p=2$ case of Theorem VI.1 of Melbourne and Tkocz \cite{MT21}. For any log-concave random vector $(X, Y)\in \R^2$, it was proved in \cite{Li18} that
$$
e^{h_2(X+Y)}\leq e^{h_2(X)}+e^{h_2(Y)},
$$
which implies the one-dimensional case of \eqref{eq:2-ent-jump} as long as $X, Y$ are jointly log-concave (even not necessarily independent) and have equal entropy $h_2(X)=h_2(Y)$.
\end{rmk}

\begin{thm}\label{thm:infty-ent-jump}
Let $X$ and $Y$ be independent real-valued log-concave random variables. Let $Z$ and $W$ be independent exponential random variables such that $h_\infty(Z)=h_\infty(X)$ and $h_\infty(W)=h_\infty(Y)$. We have
\begin{align}\label{eq:infty-ent-jump-a}
h_\infty(X+Y)\leq h_\infty(Z+W).
\end{align}
Consequently, for i.i.d. real-valued log-concave random variables $X$ and $Y$, we have
\begin{align}\label{eq:infty-ent-jump-b}
h_\infty(X+Y)\leq h_\infty(X)+1.
\end{align}
\end{thm}

The decreasing rearrangement operation tends to increase concentration on one side while reducing the additive overlap of translates in the convolution. Consequently, one expects $f^\downarrow\ast g^\downarrow$ to be less concentrated than $f\ast g$. This behavior is illustrated by the following result, which can be viewed as a reverse of Riesz's rearrangement inequality. This result was independently obtained by James Melbourne (private communication).

\begin{lem}\label{lem:unimodal}
Let $f, g: \R\rightarrow \R_+$ be integrable unimodal functions. Then it holds that
$$
\|f\ast g\|_\infty\geq \|f^\downarrow\ast g^\downarrow\|_\infty.
$$
\end{lem}

We need the following intermediate result.

\begin{lem}\label{lem:intersection}
Let $f: \mathbb{R}\rightarrow \mathbb{R}_+$ be an integrable unimodal function. For every $t>0$, there exists an interval $I$ of length $t$ such that for any $s>0$,
\begin{equation}\label{eq:intersection}
|I\cap\{f>s\}|=\min\{t,|\{f>s\}|\}.
\end{equation}
\end{lem}

\begin{proof}
For any $s, t>0$, we write
$$
m_{f}(s)=|\{x\in \mathbb{R}: f(x)>s\}|, \quad \lambda_t= \inf\{s>0: m_{f}(s)\leq t\}.
$$
Since $f$ is nonnegative and integrable, Morkov's inequality yields $m_{f}(s)\leq \|f\|_1/s<\infty$. Consequently, $0\leq \lambda_{t}<\infty$. We first show that 
\begin{equation}\label{eq:suplevel-relation}
|\{f>\lambda_t\}|\leq t \leq |\{f\geq\lambda_t\}|. 
\end{equation}
The first inequality follows from the definition of $\lambda_t$. If $\lambda_{t}=0$, we have $\{f\geq\lambda_t\}=\mathbb R$ and the second inequality trivially holds. If $\lambda_{t}>0$, we have $\{f\geq\lambda_{t}\}=\cap_{n=1}^{\infty}\left\{f>n\lambda_{t}/(n+1)\right\}$. Since $|\{f>\lambda_{t}/2\}|<\infty$, by the continuity from above of the Lebesgue measure, we have
$$
|\{f\geq\lambda_{t}\}|=\lim_{n\to \infty}|\{f>n\lambda_{t}/(n+1)\}|\geq t.
$$
The last inequality follows from that $m_{f}(s) >t$ for all $0\le s<\lambda_t$.

Now we show the existence of the desired interval $I$. Since $f$ is unimodal, $\{f>\lambda_{t}\}$ and $\{f\geq\lambda_{t}\}$ are intervals. Owing to \eqref{eq:suplevel-relation}, there exists an interval $I$ of length $t$ such that $\{f>\lambda_{t}\}\subseteq I \subseteq \{f\geq\lambda_{t}\}$.
For $s< \lambda_{t}$, we have $I \subseteq \{f\geq\lambda_{t}\}\subset \{f>s\}$ and hence
$$
|I\cap \{f>s\} |=|I|=t.
$$
For $s\geq \lambda_{t}$, we have $\{f>s\} \subseteq \{f>\lambda_{t}\}\subseteq I$ and hence
$$
|I\cap \{f>s\} |=|\{f>s\}|.
$$
This verifies identity \eqref{eq:intersection}.
\end{proof}

\begin{proof}[Proof of Lemma \ref{lem:unimodal}]
For any $t>0$,  we choose intervals $I=[a,a+t]$ for $f$ and $J=[b,b+t]$
for $g$ as in the Lemma \ref{lem:intersection}. Setting $x=a+b+t$, we have $I=x-J$. For any $s,u>0$, write $A_{s}=\{f>s\}$, $B_{u}=\{g>u\}$. Using the layer-cake formula and Tonelli's theorem, we obtain
\begin{align}
(f*g)(x)&=\int_{\mathbb{R}} f(y)g(x-y) dy \notag\\
&=\int_{\mathbb{R}}\int_{0}^{\infty}\int_{0}^{\infty} \mathbf{1}_{\{f>s\}}(y)\mathbf{1}_{\{g>u\}}(x-y) dsdudy \notag\\
&= \int_{0}^{\infty}\int_{0}^{\infty}|A_{s}\cap (x-B_{u})|dsdu \notag\\
&\geq \int_{0}^{\infty}\int_{0}^{\infty} |I\cap A_{s}\cap (x-B_{u})|dsdu.\label{eq:convlution-fg}
\end{align}
Lemma \ref{lem:intersection} gives
\begin{align*}
\alpha_{s}:= &|I\cap A_{s}| =\min\{t, |A_{s}|\},\\
\beta_{u}:= &|I\cap (x-B_{u})| =|(x-J)\cap (x-B_{u})|=|J\cap B_{u}|=\min\{t, |B_{u}|\}.
\end{align*}
Then we have
\begin{align*}
|I\cap A_{s}\cap (x-B_{u})|&=|I\cap A_{s}|+|I\cap (x-B_{u})|-|(I\cap A_{s})\cup (I\cap (x-B_{u}))|\notag\\
&\geq |I\cap A_{s}|+|I\cap (x-B_{u})|-|I|\notag\\
&=\alpha_{s}+\beta_{u}-t.
\end{align*}
Since $|I\cap A_{s}\cap (x-B_{u})|\ge0$, we obtain
\begin{equation}\label{eq:overlap-bound}
|I\cap A_{s}\cap (x-B_{u})|\ge (\alpha_{s}+\beta_{u}-t)_+.
\end{equation}
On the other hand, since $A_{s}^{\downarrow} \cap (t-B_{u}^{\downarrow})=[0,|A_{s}|)\cap (t-|B_{u}|, t] \subseteq [0,t]$, we have 
\begin{align}\label{eq:rearrangement-overlap-bound}
|A_{s}^{\downarrow} \cap (t-B_{u}^{\downarrow})| &=\big|[0,t]\cap [0,|A_{s}|)\cap (t-|B_{u}|, t]\big|\notag\\
&=\big|[0,\alpha_{s})\cap (t-\beta_{u}, t]\big|\notag\\
&=(\alpha_{s}+\beta_{u}-t)_{+}.
\end{align}
Combining \eqref{eq:convlution-fg}, \eqref{eq:overlap-bound} and \eqref{eq:rearrangement-overlap-bound}, we obtain 
\begin{align}
(f*g)(x)&\geq \int_{0}^{\infty}\int_{0}^{\infty} (\alpha_{s}+\beta_{u}-t)_{+}dsdu\notag\\
&=\int_{0}^{\infty}\int_{0}^{\infty} |A_{s}^{\downarrow} \cap (t-B_{u}^{\downarrow})|dsdu\notag\\
&=\int_{0}^{\infty}\int_{0}^{\infty}|\{f^\downarrow>s\}\cap (t-\{g^\downarrow >u\})|dsdu\notag\\
&=\int_{\mathbb{R}}\int_{0}^{\infty}\int_{0}^{\infty} \mathbf{1}_{\{f^\downarrow>s\}}(y)\mathbf{1}_{\{g^\downarrow>u\}}(t-y) dsdudy \notag\\
&=(f^\downarrow*g^\downarrow)(t). \label{eq:pointwise convolution}
\end{align}
Note the convolution of two non-negative integrable functions is lower semicontinuous. Thereofore, we have 
$$
\operatorname*{ess\,sup}_{x\in\mathbb{R}}(f*g)(x)=\sup_{x\in\mathbb{R}}(f*g)(x),\qquad  \operatorname*{ess\,sup}_{t>0}(f^{\downarrow}*g^{\downarrow})(t)=\sup_{t>0}(f^{\downarrow}*g^{\downarrow})(t).
$$
Combining this with \eqref{eq:pointwise convolution}, we obtain Lemma \ref{lem:unimodal}.
\end{proof}

\begin{rmk}
Our proofs of Lemmas \ref{lem:unimodal} and \ref{lem:intersection} crucially rely on the interval structure of superlevel sets of $f$ and $g$, which is guaranteed by their unimodality. The result may fail for arbitrary functions that lack this interval structure. For example, define $f=\mathbf{1}_{[0, 2]}$ and $g=\mathbf{1}_{[0, 1]}+\mathbf{1}_{[3, 4]}$. Clearly, $f^\downarrow=g^\downarrow=\mathbf{1}_{[0, 2]}$. One can check that $\|f\ast g\|_\infty=1$ while $\|f^\downarrow\ast g^\downarrow\|_\infty=2$.

It is natural to ask whether such a comparison holds for other $L^p$ norms. The following example shows that analogue of Lemma \ref{lem:unimodal} fails for $L^p$ norms with $1<p<2$, even assuming log-concavity. (But it is still possible that such a comparison holds for the whole class of log-concave functions when $p$ is large enough). Take $f(x)=e^{-x}\mathbf{1}_{\mathbb{R}_+}(x)$ and $g(x)=e^x\mathbf{1}_{\mathbb{R}_-}(x)$. We have $f^\downarrow=g^\downarrow=f$, 
$(f\ast g)(x)=\frac{1}{2}e^{-|x|}$ and $(f^\downarrow\ast g^{\downarrow})(x)=xe^{-x}\mathbf{1}_{\mathbb{R}_+}(x)$. One can check that
$$
\frac{\|f\ast g\|_p^p}{\|f^\downarrow\ast g^\downarrow\|_p^p}=\frac{p^p}{2^{p-1}\Gamma(p+1)}<1,\qquad (1<p<2).
$$
\end{rmk}

\begin{lem}\label{lem:exponentialization}
Let $X$ and $Y$ be independent monotone log-concave random variables. Let $Z$ and $W$ be independent exponential random variables such that $h_\infty(Z)=h_\infty(X)$ and $h_\infty(W)=h_\infty(Y)$. For all $0<p\leq\infty$, we have
\begin{align}\label{eq:major-exp}
h_p(X+Y)\leq h_p(Z+W).
\end{align}
\end{lem}

\begin{proof}
We will show the majorization of $X+Y$ over $X+W$ (i.e., the density of $X+Y$ majorizes that of $X+W$) and, in the same vein, the majorization of $X+W$ over $Z+W$. Hence, $Z+W$ is majorized by $X+Y$. This, together with inequality \eqref{eq:maj}, yields the statement \eqref{eq:major-exp}. 

Now, we show the majorization of $X+Y$ over $X+W$. Note that $X+Y$ and $X+W$ are log-concave. By Proposition \ref{prop:maj-suf}, it suffice to show that for any $[a, a+\delta]\subseteq [0, \infty)$, there exists $[b, b+\delta]\subseteq [0, \infty)$ such that
\begin{equation}\label{eq:maj-prob}
\P(X+W \in [a, a+\delta])\leq  \P(X+Y \in [b, b+\delta]).
\end{equation}
Let $f_X$ and $f_Y$ be the densities of $X$ and $Y$, respectively. We can perturb $f_X$ and $f_Y$ if necessary such that both of them have infinite support. We can also make shifts if necessary so that $\|f_X\|_\infty=f_X(0)$ and $\|f_Y\|_\infty=f_Y(0)$. Write $f_Y(x)=f_Y(0)e^{-V(x)}$ for $x\geq 0$, where $V(x)$ is an non-decreasing convex function with $V(0)=0$. Then $W$ has density $f_W(x)=f_Y(0)e^{-f_Y(0)x}$ for $x\geq0$. Consider the map $\phi =F_W^{-1}\circ F_Y$, where $F_Y$ and $F_W$ are the cumulative distribution functions of $Y$ and $W$, respectively. One can check that $W\overset{d}{=}\phi(Y)$ and $\phi(y)=-f_Y(0)^{-1}\log(1-F_Y(y))$. Then we can rewrite inequality \eqref{eq:maj-prob} as
$$
\P(X+\phi(Y) \in [a, a+\delta])\leq  \P(X+Y \in [b, b+\delta]).
$$
By conditioning on $Y$, it suffices to show for all $y\geq0$ that
\begin{align}\label{eq:major-pt}
\P(X\in [a, a+\delta]-\phi(y))\leq  \P(X\in [b, b+\delta]-y).
\end{align}
Next, we demonstrate that this inequality holds for $b=\phi^{-1}(a)$. 
Since $Y$ is log-concave, the tail probability $1-F_Y(y)$ is also log-concave. Hence, $\phi$ is a convex map. Furthermore, we have
\begin{align}\label{eq:deriv-T}
\phi'(y)=\frac{e^{-V(y)}}{1-F_Y(y)}.
\end{align}
Since $\phi$ is convex, $\phi'$ is nondecreasing, thus $\phi'(y) \geq \phi'(0)=1$.

\textit{Case 1}:  $y\geq \phi^{-1}(a+\delta)$. Since $X$ is non-negative, inequality \eqref{eq:major-pt} trivially holds.

\textit{Case 2}: $\phi^{-1}(a)<y< \phi^{-1}(a+\delta)$. By integrating $\phi'$ over $[\phi^{-1}(a),y]$,  $\phi' \geq 1$ implies that
$y-\phi^{-1}(a)\leq \phi(y)-a$, which is equivalent to $a+\delta-\phi(y)\leq \phi^{-1}(a)+\delta-y$. Since $X$ is non-negative, we obtain
$$
\P(X\in [0, a+\delta-\phi(y)]\leq  \P(X\in [0, \phi^{-1}(a)+\delta-y]),
$$
which is inequality \eqref{eq:major-pt}. 

\textit{Case 3}: $y\leq \phi^{-1}(a)$. Similar to Case 2, we can apply $\phi' \geq 1$  to obtain $a-\phi(y)\geq \phi^{-1}(a)-y\geq 0$ by integrating $\phi'$ over $[y, \phi^{-1}(a)]$.  This, together with the monotonicity of the density of $X$, yields inequality \eqref{eq:major-pt}.
\end{proof}

\begin{rmk}
The principle underpins Lemma \ref{lem:exponentialization} is that expansion of the transport map implies majorization of one measure by another. This principle was studied in more general context in \cite{MR23} (Lemma 1.5). However, it seems that the result does not directly apply to our problem.

It is natural to consider extensions of Lemma \ref{lem:exponentialization} for other orders of R\'enyi entropies. The following example shows that analogue of Lemma \ref{lem:exponentialization} fails for the Boltzmann-Shannon entropy. (But it is still possible that extensions of Lemma \ref{lem:exponentialization} holds for R\'enyi entropies of sufficiently large orders). Let $E$ be a standard exponential random variable, and let $U$ be a uniform random variable on $[0,1]$. We take all relevant variables below independently. Set
$$
X\stackrel d=E,\qquad Y_\varepsilon \stackrel d=\varepsilon U,\qquad
Z\stackrel d=E,\qquad W_\varepsilon \stackrel d=(\varepsilon/e)E.
$$
It is clear that
$$
h(X)=h(Z)=1,\qquad
h(Y_\varepsilon)=h(W_\varepsilon)=\log\varepsilon.
$$
For sufficiently small $\varepsilon>0$, we will show that
\begin{equation}\label{eq:counterexample-gap}
 h(X+Y_\varepsilon)>h(Z+W_\varepsilon).
\end{equation}
The density of $X+Y_\varepsilon$ is
$$
\phi_\varepsilon(x)=
\begin{cases}
(1-e^{-x})/\varepsilon, & 0\le x\le\varepsilon,\\[2pt]
((1-e^{-\varepsilon})/\varepsilon)\cdot e^{-(x-\varepsilon)}, & x\ge\varepsilon.
\end{cases}
$$
Write $a_\varepsilon(x)=(1-e^{-\varepsilon x})/\varepsilon$. Direct calculations give
\begin{equation}\label{eq:uniform-exact-entropy}
 h(X+Y_\varepsilon)=-\varepsilon\int_0^1a_\varepsilon(x)\log a_\varepsilon(x)\,dx+a_\varepsilon(1)(1-\log a_\varepsilon(1)).
\end{equation}
Note $a_\varepsilon(x)\to x$ as $\varepsilon\to 0$, and $0\le a_\varepsilon(u)\le1$ for $0<\varepsilon\le1$. We have
$$
a_\varepsilon(1)=1-\frac\varepsilon2+O(\varepsilon^2),\qquad
a_\varepsilon(1)(1-\log a_\varepsilon(1))=1+O(\varepsilon^2).
$$
Since $-x\log x$ is non-negative and bounded on $[0,1]$ and $-\int_0^1x\log x\,dx=1/4$, we apply Lebesgue's dominated convergence theorem to \eqref{eq:uniform-exact-entropy} and obtain
\begin{equation}\label{eq:uniform-expansion}
 h(X+Y_\varepsilon)=1+\frac{\varepsilon}{4}+o(\varepsilon).
\end{equation}
The density of $Z+W_\varepsilon$ is
$$
\psi_\varepsilon(x)=\frac{e^{-x}-e^{-(e/\varepsilon)x}}{1-\varepsilon /e},\qquad x\ge0.
$$
Write $k(u)=-(1-e^{-u})\log(1-e^{-u})$. The function $k$ is integrable and one can check that
\begin{align*}
\int_0^\infty k(u)\,du=\frac{\pi^2}{6}-1.
\end{align*}
Make the chage of variable $u=(e/\varepsilon-1)x$. We obtain
\begin{equation*}
h(Z+W_\varepsilon)=1+\frac{\varepsilon}{e}+\log\left(1-\frac{\varepsilon}{e}\right)+\frac{e\varepsilon}{(e-\varepsilon)^2}\int_0^\infty e^{-\varepsilon u/(e-\varepsilon)}k(u)\,du.
\end{equation*}
Dominated convergence then gives
\begin{equation*}
 h(Z+W_\varepsilon)
 =1+\frac{\pi^2/6-1}{e}\,\varepsilon+o(\varepsilon).
\end{equation*}
Combined with \eqref{eq:uniform-expansion}, this gives
$$
h(X+Y_\varepsilon)-h(Z+W_\varepsilon)=\left(\frac14-\frac{\pi^2/6-1}{e}\right)\varepsilon +o(\varepsilon).
$$
One can check that the coefficient is positive. This proves \eqref{eq:counterexample-gap} for sufficiently small $\varepsilon$.
\end{rmk}

\begin{proof}[Proof of Theorem \ref{thm:infty-ent-jump}]
Let $f_X$ and $f_Y$ be the densities of log-concave random variables $X$ and $Y$, respectively. As shown in Proposition \ref{prop:rearrange propt-function}, log-concavity is preserved by decreasing rearrangement, thus $f_X^\downarrow$ and $f_Y^\downarrow$ are also log-concave. Since decreasing rearrangement preserves the $L^p$ moments, $f_X^\downarrow$ and $f_Y^\downarrow$ are also probability densities such that $\|f_X^\downarrow\|_\infty=\|f_X\|_\infty$ and $\|f_Y^\downarrow\|_\infty=\|f_Y\|_\infty$. Let $X^{\downarrow}$ and $Y^{\downarrow}$ be random variables with densities $f_X^{\downarrow}$ and $f_Y^{\downarrow}$, respectively. Then we have
	$$
	h_{\infty}(X^{\downarrow})=h_{\infty}(X)=h_{\infty}(Z),~~~~ h_{\infty}(Y^{\downarrow})=h_{\infty}(Y)=h_{\infty}(W).
	$$
	By Lemma  \ref{lem:unimodal}, we have
	$$
	h_{\infty}(X+Y)\leq h_{\infty}(X^{\downarrow}+Y^{\downarrow}).
	$$
	By Lemma \ref{lem:exponentialization}, we have
	$$
	h_{\infty}(X^{\downarrow}+Y^{\downarrow}) \leq h_{\infty}(Z+W).
	$$
	Then we can combine these inequalities to obtain inequality \eqref{eq:infty-ent-jump-a}. Inequality \eqref{eq:infty-ent-jump-b} follows from simple calculations.
\end{proof}

\begin{rmk}
We give a short proof of the $p=\infty$ case of inequality \eqref{eq:major-exp} via measure transportation (in fact the change of variable). (This, together with Lemma \ref{lem:unimodal}, yields a short proof of Theorem \ref{thm:infty-ent-jump}). Since the map $\phi$ transports $Y$ to $W$, it holds for any function $\varphi$ integrable with respect to $F_W$ that
$$
\int_{\R_+} \varphi(\phi(x))f_Y(x)dx=\int_{\R_+} \varphi(x)f_W(x)dx.
$$
For any $y\geq 0$, we apply this identity to $\varphi(x)=f_X(y-x)$ for $0\leq x\leq y$ to obtain
\begin{align*}
(f_X\ast f_W)(y) &=\int_0^y f_X(y-x)f_W(x)dx\\
&= \int_0^{\phi^{-1}(y)} f_X(y-\phi(x))f_Y(x)dx\\
&\leq \int_0^{\phi^{-1}(y)} f_X(\phi^{-1}(y)-x)f_Y(x)dx\\
&=(f_X\ast f_Y)(\phi^{-1}(y)),
\end{align*}
where the inequality follows from the observation that $y-\phi(x)\geq \phi^{-1}(y)-x$, which is implied by \eqref{eq:deriv-T}, and the assumption that $f_X$ is non-increasing. This yields 
$$
h_\infty(X+Y)\leq h_\infty(X+W)
$$
and similarly we can obtain
$$
h_\infty(X+W)\leq h_\infty(Z+W).
$$
These two inequalities together yield the $p=\infty$ case of inequality \eqref{eq:major-exp}. One can see that the monotonicity of $f_X$ and the expansion property of the transportation map $\phi$ play the key role.
\end{rmk}

\begin{rmk}
Let $(X, Y)$ be an \textit{origin symmetric} log-concave random vector in $\R^2$. (The marginals are not necessarily independent). A result of Ball (\cite{Bal88}, Theorem 5) imply that 
$$
e^{h_\infty(X+Y)}\leq e^{h_\infty(X)}+e^{h_\infty(Y)}.
$$
This inequality can be seen as the entropic analog of Busemann's convexity theorem of intersection bodies \cite{Bus49}, which also motivates the study of reverse EPIs in \cite{BNT16, Li18}.
\end{rmk}


\section{Analogs for Interger-valued Log-concave Random Variables}


We first recall the majorization of two sequences. Let $\{f(k)\}_{k=1}^\infty$ and $\{g(k)\}_{k=1}^\infty$ be two summable non-negative sequences such that $\sum_{k=1}^\infty f(k)= \sum_{k=1}^\infty g(k)$. We say that $\{f(k)\}_{k=1}^\infty$ majorizes $\{g(k)\}_{k=1}^\infty$ if the decreasing rearrangements $\{f^\downarrow(k)\}_{k=1}^\infty$ and $\{g^\downarrow(k)\}_{k=1}^\infty$ satisfy that $\sum_{k=1}^n f^\downarrow(k)\geq \sum_{k=1}^k g^\downarrow(k)$ for all $n\in\Z_+$. Karamata's inequality (or majorization inequality) says that if $\{f(k)\}_{k=1}^\infty$ majorizes $\{g(k)\}_{k=1}^\infty$ then it holds for any convex function $\varphi$ on $\R$ that
\begin{equation}\label{eq:kar}
\sum_{k=1}^\infty (\varphi\circ f)(k)\geq \sum_{k=1}^\infty (\varphi\circ g)(k).
\end{equation}

An integer-valued random variable $X$ is called \textit{log-concave} if its probability mass function $\{f(k)\}_{k\in \Z}$ is a log-concave sequence; that is, $f(k+1)^2\geq f(k)f(k+2)$ holds for all $k\in\Z$ and the support of this sequence is a contiguous interval of $\mathbb{Z}$. Similar to \eqref{eq:renyi}, the order $p$ R\'enyi entropy of $X$ is defined as
$$
H_p(X)=\frac{1}{1-p}\log \sum_{k\in \Z} f(k)^p.
$$
Particularly, we have 
$H_\infty(X) =-\log \|f\|_\infty$.

\begin{prop}\label{prop:2-conv-disct}
Let $X$ and $Y$ be i.i.d log-concave random variables on $\Z$. It holds that
\begin{equation}
H_2(X+Y)<H_2(X)+\log 2.
\end{equation}
Moreover, the inequality cannot be improved.
\end{prop}

\begin{proof}
Let $\tilde{X}$ and $\tilde{Y}$ be independent copies of $X$ and $Y$, respectively. We denote by $\{f(k)\}_{k\in\Z}$ the probability mass function of  $X-\tilde{X}$. We define the geometric sequence $\{g(k)\}_{k\in\Z_+}$ as $g(k)=f(0)\left(\frac{1-f(0)}{1+f(0)}\right)^k$. One can check that $\sum_{k=0}^\infty f(k)=\sum_{k=0}^\infty g(k)$. Note that $f(0)=g(0)$ and $\{g(k)\}_{k\in\Z_+}$ is a log-affine sequence. We have the majorization of $\{f(k)\}_{k\in\Z_+}$ over $\{g(k)\}_{k\in\Z_+}$. We apply Karamata's inequality \eqref{eq:kar} to obtain
\begin{align*}
\sum_{k\in\Z} f(k)^2 &=2\sum_{k\in\Z_+} f(k)^2-f(0)^2\geq 2\sum_{k\in\Z_+} g(k)^2-f(0)^2 \\
&=\frac{1+f(0)^2}{2}\cdot f(0)>\frac{f(0)}{2}.
\end{align*}
The above inequality, together with the symmetry of $X-\tilde{X}$, implies that
\begin{align}\label{eq:3rd-b}
H_\infty((X-\tilde{X})+(Y-\tilde{Y})) &=-\log (f\ast f)(0) =-\log \sum_{k\in\Z} f(k)^2\notag\\
&<-\log f(0)+\log 2=H_\infty(X-\tilde{X})+\log 2.
\end{align}
Observe that for any i.i.d. discrete log-concave random variables $U$ and $V$, it holds that $H_2(U)=H_\infty(U-V)$. We have
\begin{align*}
H_2(X+Y) &=H_\infty((X+Y)-(\tilde{X}+\tilde{Y}))\\
&=H_\infty((X-\tilde{X})+(Y-\tilde{Y}))\\
&< H_\infty(X-\tilde{X})+\log 2\\
&= H_2(X)+\log 2.
\end{align*}
The inequality follows from \eqref{eq:3rd-b}. The moreover part is addressed later in Remark 4.4.
\end{proof}

\begin{rmk}
Similar to the continuous case, the identity $H_2(X+Y)=H_2(X-Y)$ holds for any $\Z$-valued i.i.d log-concave random variables $X$ and $Y$. Hence, Proposition \ref{prop:2-conv-disct} is the same as the $p=2$ case of Theorem 2.11 of Melbourne and Palafox-Castillo \cite{MPC23}. 
\end{rmk}

\begin{thm}\label{thm:infty-conv-disct}
Let $X$ and $Y$ be i.i.d. log-concave random variables on $\Z$. Assume that their probability mass function $f$ is monotone. It holds that
\begin{equation*}\label{eq:infty-conv-disct}
\|f\ast f\|_\infty> e^{-1}\|f\|_\infty.
\end{equation*}
Moreover, the inequality can not be improved. In other words, we have
$$
H_\infty(X+Y)<H_\infty(X)+1.
$$
\end{thm}

\begin{proof}
We assume that $f$ is non-increasing and  supported on $\Z_+$. Write $f(k)=e^{-V(k)}$, where $V$ is non-decreasing and convex. Hence, $\|f\|_\infty=e^{-V(0)}$. We apply the convexity of $V$ to obtain
\begin{align*}
(f\ast f)(k) 
&=\sum_{i=0}^ke^{-[V(i)+V(k-i)]}\geq \sum_{i=0}^ke^{-[V(0)+V(k)]}=(k+1)e^{-V(k)}\cdot \|f\|_\infty.
\end{align*}
Write $\phi(k):=(k+1)e^{-V(k)}=e^{\log(k+1)-V(k)}$. It suffices to show that
\begin{equation}\label{eq:max-disct}
\sup_{k\in\Z_+}\phi(k)>1/e.
\end{equation}
Since $\phi$ is log-concave, the supremum in \eqref{eq:max-disct} can be achieved, say at $k=k^*$. Consider the following scenarios.

\textit{Case 1}: $k^*=0$. Then it is necessary to have $\phi(0)\geq \phi(1)$, which can be rewritten as
$$
V(1)-V(0)\geq \log 2.
$$
Since $V$ is convex, it holds that
$$
V(k)\geq V(0)+k(V(1)-V(0))\geq V(0)+k\log 2.
$$
As a consequence, we obtain
$$
1=\sum_{k=0}^\infty f(k)=e^{-V(0)}\sum_{k=0}^\infty e^{-(V(k)-V(0))}\leq \phi(0)\sum_{k=0}^\infty 2^{-k}=2\phi(0). 
$$
Hence, it is necessary to have $\phi(0)\geq 1/2$ and inequality \eqref{eq:max-disct}  holds.

\textit{Case 2}: $k^*\geq 1$. Then it is necessary to have $\phi(k^*-1)\leq \phi(k^*)$ and $ \phi(k^*+1) \leq \phi(k^*)$, which can be rewritten as
\begin{align*}
V(k^*)-V(k^*-1)&\leq \log(1+1/k^*)\\
V(k^*+1)-V(k^*)&\geq \log(1+1/(k^*+1)).
\end{align*}
Owing to the convexity of $V$, we obtain for $k\leq k^*$ that
\begin{align}\label{eq:est-V-a}
V(k)&\geq V(k^*)-(k^*-k)[V(k^*)-V(k^*-1)]\notag\\
&\geq V(k^*)+(k-k^*)\log(1+1/k^*)
\end{align}
and for $k\geq k^*$ that
\begin{align}\label{eq:est-V-b}
V(k)&\geq V(k^*)+(k-k^*)[V(k^*+1)-V(k^*)]\notag\\
&\geq V(k^*)+(k-k^*)\log(1+1/(k^*+1)).
\end{align}
Apply the estimates \eqref{eq:est-V-a} and \eqref{eq:est-V-b} to obtain
\begin{align*}
1=\sum_{k=0}^\infty f(k)&\leq e^{-V(k^*)}\left(\sum_{k=0}^{k^*}\left(1+\frac{1}{k^*}\right)^{k^*-k}+\sum_{k=k^*+1}^\infty\left(1+\frac{1}{k^*+1}\right)^{-(k-k^*)}\right)\\
&=e^{-V(k^*)}\left(1+k^*\left(1+\frac{1}{k^*}\right)^{k^*+1}\right)\\
&=\phi(k^*)\left(\frac{1}{k^*+1}+\left(1+\frac{1}{k^*}\right)^{k^*}\right)\\
&<e\phi(k^*).
\end{align*}
In the last inequality, we use the fact that $(x+1)^{-1}+(1+1/x)^x\nearrow e$ as $x\rightarrow\infty$. Hence, it is necessary to have $\phi(k^*)>1/e$, i.e., inequality \eqref{eq:max-disct} holds.
\end{proof}

\begin{rmk}
Proposition \ref{prop:2-conv-disct} and Theorem \ref{thm:infty-conv-disct} are tight for geometric distributions $f(k)=(1-\lambda)\lambda^k$ for $k\in\mathbb{Z}_+$ and $0<\lambda<1$. It is clear that $\|f\|_\infty=1-\lambda$. One can check that $(f\ast f)(k)=(k+1)(1-\lambda)^2\lambda^k$. Therefore, we have
$$
e^{H_2(f\ast f)-H_2(f)}=\frac{\sum_{k=0}^\infty f(k)^2}{\sum_{k=0}^\infty (f\ast f)(k)^2}=\frac{(1+\lambda)^2}{1+\lambda^2}\rightarrow 2~~\text{as}~\lambda\rightarrow 1.
$$
Next, we show the tightness of Theorem \ref{thm:infty-conv-disct}. Set $\lambda_{n}=e^{-1/(n+1)}$. Consider the corresponding probability mass function by $f_{n}(k)=(1-\lambda_{n})\lambda_{n}^{k}$. Using $e^{-x}>1-x$ for $x>0$, we obatin 
$$
\frac{(f_n*f_n)(n)}{(f_n*f_n)(n-1)}=\lambda_n\cdot \frac{n+1}{n}
>\frac{n}{n+1}\cdot \frac{n+1}{n}=1.
$$
On the other hand, since  $e^{-x}<1/(1+x)$ for $x>0$, we have 
$$
\frac{(f_n*f_n)(n+1)}{(f_n*f_n)(n)}=\lambda_n\cdot\frac{n+2}{n+1}
<\frac{n+1}{n+2}\cdot \frac{n+2}{n+1}=1.
$$
Therefore, $f_n\ast f_n$ achieves a local maximum at $n$, and it has to be also a global maximum via the log-concavity/unimodality of $f_n\ast f_n$. 
Consequently, we obtain
$$
\frac{\|f_n*f_n\|_\infty}{\|f_n\|_\infty}=\frac{(f_n*f_n)(n)}{f_n(1)}=(n+1)(1-\lambda_n)\lambda_n^n\rightarrow \frac{1}{e}\quad \text{as} \quad n\to \infty.
$$
\end{rmk}

{\bf Acknowledgement.} J. L. would like to thank Mokshay Madiman and James Melbourne for valuable discussions. The work is supported by the National Natural Science Foundation of China (NSFC) grant 62201175.


\bibliographystyle{plain}

\end{document}